\documentclass[a4paper,10pt]{article}

\usepackage{geometry}
\usepackage{epsfig}
\usepackage{amsmath}
\usepackage{amssymb}
\usepackage{amsthm}
\usepackage{mathrsfs}
\usepackage{color}
\usepackage{amscd}

\usepackage{amsthm}
\newtheorem{theorem}{Theorem}[section]
\newtheorem{definition}[theorem]{Definition}
\newtheorem{lemma}[theorem]{Lemma}
\newtheorem{proposition}[theorem]{Proposition}
\newtheorem{corollary}[theorem]{Corollary}
\newtheorem{remark}[theorem]{Remark}
\newtheorem{example}[theorem]{Example}
\newtheorem{examples}[theorem]{Examples}

\usepackage{amsfonts}

\newcommand{\oo}{{\mathbb{O}}}
\newcommand{\hh}{{\mathbb{H}}}
\newcommand{\cc}{{\mathbb{C}}}
\newcommand{\rr}{{\mathbb{R}}}
\newcommand{\zz}{{\mathbb{Z}}}
\newcommand{\nn}{{\mathbb{N}}}

\newcommand{\s}{{\mathbb{S}}}
\newcommand{\sr}{\mathcal{SR}}

\newcommand{\I}{\mathcal{I}}
\renewcommand{\P}{\mathcal{P}}

\newcommand{\B}{\mathcal{B}}
\newcommand{\T}{\mathcal{T}}

\newcommand{\h}{{\bf h}}
\renewcommand{\k}{{\bf k}}
\newcommand{\F}{\mathscr{F}}

\newcommand{\debar}{\overline{\partial}}

\newcommand{\Span}{\operatorname{Span}}

\newcommand{\torus}{\mathbb{T}}
\newcommand{\mon}{\mathrm{Mon}}
\newcommand{\reg}{\mathrm{Reg}}
\newcommand{\slice}{\mathcal{S}}

\newcommand{\mr}{\mathrm}
\newcommand{\mscr}{\mathscr}

\newcommand{\hslashslash}{%
  \raisebox{.9ex}{%
    \scalebox{.7}{%
      \rotatebox[origin=c]{18}{$-$}%
    }%
  }%
}
\newcommand{\fslash}{%
  {%
   \vphantom{f}%
   \ooalign{\kern.05em\smash{\hslashslash}\hidewidth\cr$f$\cr}%
   \kern.05em
  }%
}

\title{\bf A unified notion of regularity in one hypercomplex variable}

\author{Riccardo Ghiloni\thanks{Partly supported by: GNSAGA INdAM; Progetto ``Teoria delle funzioni ipercomplesse e applicazioni'' Universit\`a di Firenze; PRIN 2017 ``Moduli theory and birational classification'' MIUR.}\\
\small Dipartimento di Matematica, Universit\`a di Trento\\ 
\small Via Sommarive 14, I-38123 Povo Trento, Italy\\
\small riccardo.ghiloni@unitn.it\\
\and
Caterina Stoppato\thanks{Partly supported by: GNSAGA INdAM; Progetto ``Teoria delle funzioni ipercomplesse e applicazioni'' Universit\`a di Firenze; PRIN 2022 ``Real and complex manifolds: geometry and holomorphic dynamics'' MIUR and by Finanziamento Premiale FOE 2014 ``Splines for accUrate NumeRics: adaptIve models for Simulation Environments'' INdAM.}
\\ 
\small Dipartimento di Matematica e Informatica ``U. Dini'', Universit\`a di Firenze \\
\small Viale Morgagni 67/A, I-50134 Firenze, Italy\\
\small caterina.stoppato@unifi.it}

\date{  }

\begin{document}

\maketitle


\begin{abstract}
We define a very general notion of regularity for functions taking values in an alternative real $*$-algebra. Over Clifford numbers, this notion subsumes the well-established notions of monogenic function and slice-monogenic function. Over quaternions, in addition to subsuming the notions of Fueter-regular function and of slice-regular function, it gives rise to an entirely new theory, which we develop in some detail.
\end{abstract}


\section{Introduction}\label{sec:introduction}

The aim of this work is two-fold. On the one hand, it aims at defining a unified notion of regularity for functions taking values in an alternative real $*$-algebra $A$, general enough to subsume some of the best-known notions of regularity in one hypercomplex variable. On the other hand, it aims at finding, within the scope of this new notion, some interesting new function theory.

The first aim is addressed in Section~\ref{sec:hypercomplexsubspaces}, which sets up appropriate domains $\Omega$ within the $*$-algebra $A$, and in Section~\ref{sec:hypercomplexregularity}, which defines \emph{$T$-regular functions} $f:\Omega\to A$. When $A$ is the real Clifford algebra $C\ell(0,n)$, $T$-regularity subsumes both the notions of monogenic function (see, e.g.,~\cite{librosommen,librocnops,librogurlebeck2}) and of slice-monogenic function (\cite{israel,librodaniele2}). When $A$ is the real algebra of quaternions $\hh$, $T$-regularity not only subsumes the theories of Fueter (see~\cite{fueter1,fueter2,sudbery}), of Gentili-Struppa (see~\cite{cras,advances} and~\cite{librospringer2}) and of Moisil-Teodorescu (see~\cite{moisilteodorescu}), but also includes the brand new theory of \emph{$(1,3)$-regular} quaternionic functions. When $A$ is the real algebra of octonions $\oo$, the concept of $T$-regular function subsumes the notions of octonionic monogenic function (based on~\cite{sce} and considered in the recent work~\cite{krausshardifferentialtopological}) and of slice-regular function (see~\cite{rocky}).

Section~\ref{sec:(1,3)-regularity} studies in some detail the new theory of $(1,3)$-regular quaternionic functions, which turns out to be very interesting and motivates a forthcoming work on general $T$-regular functions. 


\section{Hypercomplex subspaces}\label{sec:hypercomplexsubspaces}

Let $(A,+,\cdot,^c)$ be a real $*$-algebra of finite dimension, i.e., a finite-dimensional $\rr$-vector space endowed with an $\rr$-bilinear multiplicative operation and with an involutive $\rr$-linear antiautomorphism $x\mapsto x^c$ (called $*$-involution). Assume $A$ to be alternative, i.e., assume $x(xy)=x^2y,(xy)y=xy^2$ for all $x,y\in A$, which is automatically true if $A$ is associative. Set $t(x):=x+x^c$ and $n(x):=xx^c$ for all $x\in A$. The \emph{quadratic cone} was defined in~\cite{perotti} by means of the equality $Q_A:=\rr\cup\{x\in A\setminus\rr:t(x)\in\rr,n(x)\in\rr,4n(x)>t(x)^2\}$ and has the property
\[Q_A=\bigcup_{J\in\s_A}\cc_J\,,\]
where $\s_A=\{x\in A: t(x)=0,n(x)=1\}$ and where $\cc_J:=\rr+J\rr$ for all $J\in\s_A$ is $*$-isomorphic to $\cc$. In particular, if $x=\alpha+\beta J\in Q_A$ (with $\alpha,\beta\in\rr,J\in\s_A$) then: the conjugate $x^c=\alpha-\beta J$ belongs to $Q_A$; $t(x)=2\alpha\in\rr$; $n(x)=n(x^c)=\alpha^2+\beta^2\in\rr$; provided $x\neq0$, the element $x$ has a multiplicative inverse, namely $x^{-1}=n(x)^{-1}x^c=x^cn(x)^{-1}$, which still belongs to $Q_A$. We assume $\s_A\neq\emptyset$, whence $\rr\subsetneq Q_A$. Some examples follow, see~\cite{ebbinghaus,librogurlebeck2} for full details.

\begin{examples}
The division algebras $\cc=\rr+i\rr$ of complex numbers, $\hh=\rr+i\rr+j\rr+k\rr$ of real quaternions and $\oo=\rr+i\rr+j\rr+k\rr+l\rr+li\rr+lj\rr+lk\rr$ of real octonions are alternative real $*$-algebras of dimensions $2,4,8$, respectively. The equalities $Q_\cc=\cc,Q_\hh=\hh,Q_\oo=\oo$ hold, while $\s_\cc,\s_\hh,\s_\oo$ are, respectively, the $0,2,6$-dimensional unit spheres in the respective subspaces $t(x)=0$.
\end{examples}

For any $n\in\nn^*$, let $\mscr{P}(n)$ denote the power set of $\{1,\ldots,n\}$.

\begin{examples}
For each $n\in\nn^*$, the Clifford algebra $C\ell(0,n)$, whose elements have the form $\sum_{K\in\mscr{P}(n)}x_Ke_K$ with $x_K\in\rr$ for all $K\in\mscr{P}(n)$, is an associative real $*$-algebra of dimension $2^n$. The sets $\s_{C\ell(0,n)}$ and $Q_{C\ell(0,n)}$ are nested proper real algebraic subsets of $C\ell(0,n)$.
\end{examples}

We now focus on specific subsets of the quadratic cone $Q_A$.

\begin{definition}
Let $M$ be a real vector subspace of $A$. An ordered real vector basis $(v_0,v_1,\ldots,v_m)$ of $M$ is called a \emph{hypercomplex basis} of $M$ if: $m\geq1$; $v_0=1$; $v_s\in\s_A$ and $v_sv_t=-v_tv_s$ for all distinct $s,t\in\{1,\ldots,m\}$. The subspace $M$ is called a \emph{hypercomplex subspace} of $A$ if $\rr\subsetneq M\subseteq Q_A$.
\end{definition}

A basis $(v_0,v_1,\ldots,v_m)$ is a hypercomplex basis if, and only if, $t(v_s)=0,n(v_s)=1$ and $t(v_sv_t^c)=0$ for all distinct $s,t\in\{1,\ldots,m\}$. We point out that, for any $\ell\in\{1,\ldots,m\}$, the shortened ordered set $(v_0,v_1,\ldots,v_\ell)$ is a hypercomplex basis of its span. The concept of hypercomplex subspace was defined in~\cite[\S3]{perotticr}, which also proved the first statement in the next theorem (cf.~\cite[Lemma 1.4]{volumeintegral}).

\begin{theorem}\label{thm:hypercomplexbasis}
Every hypercomplex subspace of $A$ admits a hypercomplex basis.

Conversely, if a real vector subspace $M$ of $A$ has a hypercomplex basis $\B=(v_0,v_1,\ldots,v_m)$, then $M$ is a hypercomplex subspace of $A$. Moreover, if we complete $\B$ to a real vector basis $\B'=(v_0,v_1,\ldots,v_m,v_{m+1},\ldots,v_d)$ of $A$ and if we endow $A$ with the standard Euclidean scalar product $\langle\cdot,\cdot\rangle$ and norm $\Vert\cdot\Vert$ associated to $\B'$, then $t(xy^c)=t(yx^c)=2\langle x,y\rangle$ and $n(x)=n(x^c)=\Vert x\Vert^2$ for all $x,y\in M$.
\end{theorem}

\begin{proof}
The first statement was proven in~\cite[\S3]{perotticr}. Let $(v_0,v_1,\ldots,v_m,v_{m+1},\ldots,v_d)$ be a real vector basis of $A$. If $x=\sum_{s=0}^dx_sv_s$ and $y=\sum_{t=0}^dy_tv_t$, then
\begin{align*}
t(xy^c)&=\sum_{s=0}^dx_sy_st(n(v_s))+\sum_{s=1}^d(x_0y_s+y_0x_s)t(v_s)+\sum_{1\leq s<t\leq d}(x_sy_t+x_ty_s)t(v_sv_t^c)\,,\\
n(x)&=\sum_{s=0}^dx_s^2n(v_s)+\sum_{s=1}^dx_0x_st(v_s)+\sum_{1\leq s<t\leq d}x_sx_tt(v_sv_t^c)\,.
\end{align*}
If $(v_0,v_1,\ldots,v_m)$ is a hypercomplex basis of $M$, then $n(v_s)=1,t(n(v_s))=2$ and $t(v_s)=0=t(v_sv_t^c)$ for all distinct $s,t\in\{1,\ldots,m\}$. If $x,y\in M$ (whence $x_s=0=y_t$ for all $s,t\in\{m+1,\ldots,d\}$), we find that $t(xy^c)=2\sum_{s=0}^mx_sy_s=2\langle x,y\rangle$ and that $n(x)=\sum_{s=0}^mx_s^2=\Vert x\Vert^2$, i.e., the last statement in the theorem. We are left with proving the second statement: namely, that $M=\Span(v_0,v_1,\ldots,v_m)$ is a hypercomplex subspace of $A$. Since $v_0=1,v_1\not\in\rr$, the proper inclusion $\rr\subsetneq M$ immediately follows. The inclusion $M\subseteq Q_A$ can be proven as follows. If $x\in M\setminus\rr$, then $4n(x)=4\sum_{s=0}^mx_s^2>4x_0^2=(2x_0)^2=(x+x^c)^2=t(x)^2$, whence $x\in Q_A$.
\end{proof}

\begin{example}\label{ex:paravectors}
The subspace of paravectors $\rr^{n+1}$ is a hypercomplex subspace of the Clifford algebra $C\ell(0,n)$, having $\B=(e_\emptyset,e_1,\ldots,e_n)$ as a hypercomplex basis. If we complete $\B$ to the standard basis $\B'=(e_K)_{K\in\mscr{P}(n)}$ of $C\ell(0,n)$ and consider the standard scalar product and norm on $C\ell(0,n)$, then the equalities appearing in Theorem~\ref{thm:hypercomplexbasis} are well-known properties of paravectors.
\end{example}

\begin{examples}[{\cite[Example 1.15]{gpsalgebra}}]\label{ex:svectors}
For every $h\in\{1,\ldots,n\}$ with $h\equiv1\,\mr{mod}\,4$,
\[V_h:=\left\{x_0+\sum_{1\leq s_1<\ldots<s_h\leq n}x_{s_1\ldots s_h}e_{s_1\ldots s_h} : x_0, x_{s_1\ldots s_h}\in\rr\right\}\]
is a hypercomplex subspace of $C\ell(0,n)$, with $\B=(e_{s_1\ldots s_h})_{1\leq s_1<\ldots<s_h\leq n}$ as a hypercomplex basis and $\dim V_h=\binom{n}{h}+1\geq\left(\frac{n}{h}\right)^h+1$. If we set $h(n):=4\lfloor \frac{n+2}8\rfloor+1$ (whence $\frac{n}2-2<h(n)\leq\frac{n}2+2$), then $\dim V_{h(n)}$ grows exponentially with $n$.
\end{examples}

More associative examples can be constructed by means of the next lemma.

\begin{lemma}
Assume $A$ is associative and $M$ is a real vector subspace of $A$ with a hypercomplex basis $\B=(v_0,v_1,\ldots,v_m)$. Set $\widehat v:=v_1\cdots v_m$. Then $\widehat\B:=(v_0,v_1,\ldots,v_m,\widehat v)$ is a hypercomplex basis of $\widehat M:=\Span(\widehat\B)$ if, and only if, $m\equiv2\,\mr{mod}\,4$. If so, not only $M$ but also $\widehat M$ is a hypercomplex subspace of $A$.
\end{lemma}

\begin{proof}
Using the hypothesis that $\B$ is a hypercomplex basis, a direct computation shows that $\widehat v^c=-\widehat v$ for $m\equiv1,2\,\mr{mod}\,4$ and $\widehat v^c=\widehat v$ for $m\equiv0,3\,\mr{mod}\,4$; that $n(\widehat v)=1$; and that $v_\ell\widehat v=(-1)^{m-1}\widehat vv_\ell$ for all $\ell\in\{1,\ldots,m\}$. Overall, the desired equalities $t(\widehat v)=0,n(\widehat v)=1$ and $t(v_\ell\widehat v^c)=0$ for all $\ell\in\{1,\ldots,m\}$ hold true if, and only if, $m\equiv2\,\mr{mod}\,4$. If this is the case then, by the first part of the proof of Theorem~\ref{thm:hypercomplexbasis}, any $x=\sum_{s=0}^nx_sv_s+\widehat x\widehat v\in\widehat M$ (with $x_0,\ldots,x_n,\widehat x\in\rr$) has $n(x)=\sum_{s=0}^nx_s^2+\widehat x^2$. If $x=0$, then $n(x)=0$ and $x_0=\ldots=x_n=\widehat x=0$. We conclude that $v_0,v_1,\ldots,v_m,\widehat v$ are linearly independent, as desired.
\end{proof}

\begin{example}\label{ex:quaternions}
Examples of hypercomplex subspaces of the real algebra of quaternions $\hh=C\ell(0,2)$ are not only the subspace of paravectors, or \emph{reduced quaternions}, $\rr^{2+1}$ with hypercomplex basis $(e_\emptyset,e_1,e_2)$, but also the whole algebra $\hh$ with hypercomplex basis $(e_\emptyset,e_1,e_2,e_{12})=(1,i,j,k)$. Unless otherwise stated, we shall assume $\B=\B'=(1,i,j,k)$ and endow $\hh$ with standard scalar product and norm.
\end{example}

\begin{example}
An example of hypercomplex subspace of $C\ell(0,n)$, distinct from those considered in Examples~\ref{ex:paravectors} and~\ref{ex:svectors}, is $\Span(e_\emptyset,e_1,e_2,\ldots,e_m,e_{12\ldots m})$, for any $m\leq n$ with $m\equiv2\,\mr{mod}\,4$. 
\end{example}

We also have a nonassociative example.

\begin{example}
$\oo$ is a hypercomplex subspace of $\oo$. We shall set $\B=\B'=(1,i,j,k,l,li,lj,lk)$ and endow $\oo$ with its standard scalar product and norm.
\end{example}


\section{Regularity in hypercomplex subspaces}\label{sec:hypercomplexregularity}

In this section, we work on a fixed hypercomplex subspace of our alternative real $*$-algebra $A$, having a hypercomplex basis $\B=(v_0,v_1,\ldots,v_n)$ for some $n\in\nn^*$. After completing $\B$ to a real vector basis $\B'=(v_0,v_1,\ldots,v_n,v_{n+1},\ldots,v_d)$ of $A$, we endow $A$ with the standard Euclidean scalar product $\langle\cdot,\cdot\rangle$ and norm $\Vert\cdot\Vert$ associated to $\B'$. We denote our fixed hypercomplex subspace simply by $\rr^{n+1}$, thinking of the paravector subspace of $C\ell(0,n)$, Example~\ref{ex:paravectors}, as a guiding example for this section. We need a few further notations: for $0\leq\ell<m\leq n$, we consider the $(m-\ell+1)$-dimensional subspace
\[\rr_{\ell,m}:=\Span(v_\ell,\ldots,v_m)\,.\]
Its unit $(m-\ell)$-sphere is denoted by $\s_{\ell,m}$ and is a subset of $\s_A$ if, and only if, $\ell\geq1$. For instance: $\rr_{0,n}=\rr^{n+1}$ and $\s_{1,n}=\{\sum_{t=1}^nx_tv_t\in A:\sum_{t=1}^nx_t^2=1\}=\s_A\cap\rr^{n+1}$.

\begin{definition}
To each \emph{number of steps} $\tau\in\{0,\ldots,n\}$ and each \emph{list of steps} $T=(t_0,\ldots,t_\tau)\in\nn^{\tau+1}$, with $0\leq t_0<t_1<\ldots<t_\tau=n$, we associate the \emph{$T$-fan}
\[\rr_{0,t_0}\subsetneq\rr_{0,t_1}\subsetneq\ldots\subsetneq\rr_{0,t_\tau}=\rr^{n+1}\,.\]
The first subspace, $\rr_{0,t_0}$, is called the \emph{mirror}. We define the \emph{$T$-torus} as
\[\torus:=\s_{t_0+1,t_1}\times\ldots\times\s_{t_{\tau-1}+1,t_\tau}\]
when $\tau\geq1$ and as $\torus:=\emptyset$ when $\tau=0$.
\end{definition}

We assume henceforth $\tau\in\{0,\ldots,n\}$ and $T=(t_0,\ldots,t_\tau)\in\nn^{\tau+1}$ (with $0\leq t_0<t_1<\ldots<t_\tau=n$) are fixed. Necessarily, $n\geq t_0+\tau$. The mirror $\rr_{0,t_0}$ of the $T$-fan is either the real axis $\rr$ or a hypercomplex subspace of $A$, while all other elements of the $T$-fan are hypercomplex subspaces of $A$. Moreover, if $\tau\geq1$ then, for every $h\in\{1,\ldots,\tau\}$, the sphere $\s_{t_{h-1}+1,t_h}$ has dimension $t_h-t_{h-1}-1$, whence the $T$-torus $\torus$ is a subset of $(\s_A)^\tau$ of dimension $n-t_0-\tau$. On the other hand, $\torus=\emptyset$ if $\tau=0$.

\begin{example}
For paravectors in $C\ell(0,n)$, Example~\ref{ex:paravectors}, the $T$-fan is
\[\rr^{t_0+1}\subsetneq\rr^{t_1+1}\subsetneq\ldots\subsetneq\rr^{t_\tau+1}=\rr^{n+1}\,.\]
\end{example}

\begin{example}\label{ex:quaternions2}
In $\hh$, Example~\ref{ex:quaternions}, the $3$-fan is $\hh$, the $(2,3)$-fan is $\rr+i\rr+j\rr\subsetneq\hh$, the $(1,3)$-fan is $\cc\subsetneq\hh$, the $(0,3)$-fan is $\rr\subsetneq\hh$, the $(1,2,3)$-fan is $\cc\subsetneq\rr+i\rr+j\rr\subsetneq\hh$, the $(0,2,3)$-fan is $\rr\subsetneq\rr+i\rr+j\rr\subsetneq\hh$, the $(0,1,3)$-fan is $\rr\subsetneq\cc\subsetneq\hh$, and the $(0,1,2,3)$-fan is $\rr\subsetneq\cc\subsetneq\rr+i\rr+j\rr\subsetneq\hh$.
\end{example} 

\begin{remark}\label{rmk:decomposedvariable}
Each $x=\sum_{\ell=0}^nv_\ell x_\ell\in\rr^{n+1}$ can be decomposed as $x=x^0+x^1+\ldots+x^\tau$, where $x^h:=\sum_{\ell=t_{h-1}+1}^{t_h}x_\ell v_\ell\in\rr_{t_{h-1}+1,t_h}$ (with $t_{-1}:=-1$). This decomposition is orthogonal, whence unique. Moreover, if $\tau\geq1$ then there exist $\beta=(\beta_1,\ldots,\beta_\tau)\in\rr^\tau$ and $J=(J_1,\ldots,J_\tau)\in\torus$ such that
\begin{equation}\label{eq:decomposedvariable}
x=x^0+\beta_1J_1+\ldots+\beta_\tau J_\tau\,.
\end{equation}
This equality holds true exactly when, for each $h\in\{1,\ldots,\tau\}$: either $x^h\neq0,\beta_h=\pm\Vert x^h\Vert$ and $J_h=\frac{x^h}{\beta_h}$; or $x^h=0,\beta_h=0$ and $J_h$ is any element of $\s_{t_{h-1}+1,t_h}$.
\end{remark}

We now wish to define $J$-monogenicity. This requires a preliminary lemma.

\begin{lemma}
If $\tau\geq1$, fix $J=(J_1,\ldots,J_\tau)\in\torus$ and set
\[\rr^{t_0+\tau+1}_J:=\Span(\B_J)\,,\quad\B_J:=(v_0,v_1,\ldots,v_{t_0},J_1,\ldots,J_\tau)\,.\]
If $\tau=0$ (whence $t_0=n\geq1$), set $J:=\emptyset,\B_\emptyset:=(v_0,v_1,\ldots,v_{t_0}),\rr^{t_0+1}_\emptyset:=\Span(\B_\emptyset)=\rr^{n+1}$. 
In either case, $\B_J$ is a hypercomplex basis of $\rr^{t_0+\tau+1}_J$, which is therefore a hypercomplex subspace of $A$ contained in $\rr^{n+1}$. Moreover, if $J'\in\torus$, then the equality $\rr^{t_0+\tau+1}_J=\rr^{t_0+\tau+1}_{J'}$ is equivalent to $J'\in\{\pm J_1\}\times\ldots\times\{\pm J_\tau\}$.
\end{lemma}

\begin{proof}
If $\tau=0$, then $\B_\emptyset:=(v_0,v_1,\ldots,v_{t_0})=(v_0,v_1,\ldots,v_n)$ is a hypercomplex basis of $\rr^{t_0+1}_\emptyset=\rr^{n+1}$ by construction.

Now assume $\tau\geq1$. By construction, $v_0=1,v_1,\ldots,v_{t_0},J_1,\ldots,J_\tau\in\s_{1,n}=\s_A\cap\rr^{n+1}$ and $t(v_\ell v_m^c)=0$ for all distinct $\ell,m\in\{1,\ldots,t_0\}$. Moreover, for all $\ell\in\{1,\ldots,t_0\}$ and for all distinct $h,h'\in\{1,\ldots,\tau\}$, $v_\ell,J_h$ belong to the mutually orthogonal subspaces $\rr_{0,t_0},\rr_{t_{h-1}+1,t_h}$ and $J_h,J_{h'}$ belong to the mutually orthogonal subspaces $\rr_{t_{h-1}+1,t_h},\rr_{t_{h'-1}+1,t_{h'}}$. This proves that $\B_J$ is a basis of $\rr^{t_0+\tau+1}_J$ and yields, thanks to Theorem~\ref{thm:hypercomplexbasis}, the remaining desired equalities $t(v_\ell J_h^c)=\langle v_\ell,J_h\rangle=0$ and $t(J_hJ_{h'}^c)=\langle J_h,J_{h'}\rangle=0$. The last statement follows immediately from Remark~\ref{rmk:decomposedvariable}.
\end{proof}

\begin{definition}
If $\tau\geq1$, fix $J=(J_1,\ldots,J_\tau)\in\torus$. If $\tau=0$, set $J:=\emptyset$. For any open subset $V$ of $\rr^{t_0+\tau+1}_J$, we define the \emph{$J$-Cauchy-Riemann operator} $\debar_J:\mscr{C}^{1}(V,A)\to\mscr{C}^0(V,A)$ as $\debar_J:=2\debar_{\B_J}$, according to~\cite[Definition 2]{perotticr}. We also define the operators $\partial_J:\mscr{C}^{1}(V,A)\to\mscr{C}^0(V,A)$ and $\Delta_J:\mscr{C}^{2}(V,A)\to\mscr{C}^0(V,A)$ as $\partial_J:=2\partial_{\B_J}$ and $\Delta_J:=\Delta_{\B_J}$, according to~\cite[page 7]{perotticr}. Explicitly, referring to the decomposition~\eqref{eq:decomposedvariable} of the variable $x$, we have
\begin{align*}
\debar_J&=\partial_{x_0}+v_1\partial_{x_1}+\ldots+v_{t_0}\partial_{x_{t_0}}+J_1\partial_{\beta_1}+\ldots+J_\tau\partial_{\beta_\tau}\,,\\
\partial_J&=\partial_{x_0}-v_1\partial_{x_1}-\ldots-v_{t_0}\partial_{x_{t_0}}-J_1\partial_{\beta_1}-\ldots-J_\tau\partial_{\beta_\tau}\,,\\
\Delta_J&=\partial_{x_0}^2+\partial_{x_1}^2+\ldots+\partial_{x_{t_0}}^2+\partial_{\beta_1}^2+\ldots+\partial_{\beta_\tau}^2\,.
\end{align*}
The kernel of $\debar_J$ is denoted by $\mon_J(V,A)$ and its elements are called \emph{$J$-monogenic} functions. The elements of the kernel of $\Delta_J$ are called \emph{$J$-harmonic} functions.
\end{definition}

Using the formal~\cite[Definition 2]{perotticr} is necessary to guarantee, for $J,J'\in\torus$,
\begin{equation}\label{eq:CRwellposed}
\debar_J=\debar_{J'}\ \Longleftarrow\ \rr^{t_0+\tau+1}_J=\rr^{t_0+\tau+1}_{J'}\,.
\end{equation}
Similar considerations apply to $\partial_J,\Delta_J$. By~\cite[Proposition 5 (b)]{perotticr}, we remark:

\begin{remark}\label{rmk:J-harmonic}
Fix any open subset $V$ of $\rr^{t_0+\tau+1}_J$. The equalities $\debar_J\partial_J=\partial_J\debar_J=\Delta_J$ hold true on $\mscr{C}^{2}(V,A)$. In particular, every $\mscr{C}^{2}$ $J$-monogenic function is $J$-harmonic, whence real analytic.
\end{remark}

We now define the new concept of $T$-regular function.

\begin{definition}
If either $J\in\torus$ or $J=\emptyset$, for every $Y\subseteq\rr^{n+1},f:Y\to A$, we call $Y_J:=Y\cap\rr^{t_0+\tau+1}_J$ the \emph{$J$-slice} of $Y$ and consider the restriction $f_J:=f_{|_{Y_J}}$. We call a nonempty connected open $\Omega\subseteq\rr^{n+1}$ a \emph{domain} in $\rr^{n+1}$. A function $f:\Omega\to A$ on a domain $\Omega$ is termed \emph{$T$-regular} if, the restriction $f_J:\Omega_J\to A$ is $J$-monogenic for every $J\in\torus$, if $\tau\geq1$ (for $J=\emptyset$, if $\tau=0$). If, moreover, $f(\rr^{t_0+\tau+1}_J)\subseteq\rr^{t_0+\tau+1}_J$ for all $J\in\torus$, then $f$ is called \emph{$T$-slice preserving}. The class of $T$-regular functions $\Omega\to A$ is denoted by $\reg_T(\Omega,A)$.
\end{definition}

Over $C\ell(0,n)$, $T$-regularity subsumes some of the best-known function theories.

\begin{example}
Within the paravector subspace $\rr^{n+1}$ of $C\ell(0,n)$, Example~\ref{ex:paravectors}, fix a domain $\Omega$. For any function $f:\Omega\to C\ell(0,n)$:
\begin{itemize}
\item $f$ is $n$-regular if, and only if, it is in the kernel of the operator $\partial_{x_0}+e_1\partial_{x_1}+\ldots+e_{n}\partial_{x_n}$; this is the definition of \emph{monogenic} function (see, e.g.,~\cite{librosommen,librocnops,librogurlebeck2});
\item $f$ is $(0,n)$-regular if, and only if, for any $J_1\in\s_{1,n}=\s_{C\ell(0,n)}\cap\rr^{n+1}$, the restriction $f_{J_1}$ to the planar domain $\Omega_{J_1}\subseteq\cc_{J_1}$ is a holomorphic map $(\Omega_{J_1},J_1)\to(C\ell(0,n),J_1)$; this is the same as being \emph{slice-monogenic},~\cite{israel} (or \emph{slice-hyperholomorphic},~\cite{librodaniele2}).
\end{itemize}
\end{example}

Over the hypercomplex subspace $\hh$ of $\hh$, we achieve a complete classification of $T$-regularity. In particular, we show that $T$-regularity not only subsumes the best-known function theories, but also includes an entirely new function theory.

\begin{example}\label{ex:quaternions3}
Fix a domain $\Omega$ in $\hh$ and a function $f:\Omega\to\hh$. Then:
\begin{itemize}
\item $f$ is $3$-regular if, and only if, it belongs to the kernel of the left Cauchy-Riemann-Fueter operator $\partial_{x_0}+i\partial_{x_1}+j\partial_{x_2}+k\partial_{x_3}$; this is the definition of \emph{left Fueter-regular} function (see,~\cite{fueter1,fueter2,sudbery});
\item $f$ is $(2,3)$-regular if, and only if, $(\partial_{x_0}+i\partial_{x_1}+j\partial_{x_2}+J_1\partial_{\beta_1})f(x_0+ix_1+jx_2+\beta_1J_1)\equiv0$ for all $J_1\in\s_{3,3}=\{\pm k\}$; this is the same as being left Fueter-regular by~\eqref{eq:CRwellposed};
\item $f$ is $(1,3)$-regular if, and only if,
\[\hskip35pt\debar_{J_1}f(x_0+ix_1+\beta_1J_1):=(\partial_{x_0}+i\partial_{x_1}+J_1\partial_{\beta_1})f(x_0+ix_1+\beta_1J_1)\equiv0\]
for all $J_1$ in the $(1,3)$-torus $\s_{2,3}$, which is simply the circle $\s^1:=\s_\hh\cap(j\rr+k\rr)$; this is a brand new theory, studied in the forthcoming Section~\ref{sec:(1,3)-regularity};
\item $f$ is $(0,3)$-regular if, and only if, for any $J_1\in\s_{1,3}=\s_\hh$, the restriction $f_{J_1}$ to the planar domain $\Omega_{J_1}\subseteq\cc_{J_1}$ is a holomorphic map $(\Omega_{J_1},J_1)\to(\hh,J_1)$; this is the definition of \emph{slice-regular} function,~\cite{librospringer2} (or \emph{Cullen-regular} in the original articles~\cite{cras,advances});
\item $f$ is $(1,2,3)$-regular if, and only if, $(\partial_{x_0}+i\partial_{x_1}+J_1\partial_{\beta_1}+J_2\partial_{\beta_2})f(x_0+ix_1+\beta_1J_1+\beta_2J_2)\equiv0$ for all $(J_1,J_2)\in\s_{2,2}\times\s_{3,3}=\{\pm j\}\times\{\pm k\}$; this is the same as being left Fueter-regular by~\eqref{eq:CRwellposed};
\item $f$ is $(0,1,3)$-regular if, and only if, $(\partial_{x_0}+J_1\partial_{\beta_1}+J_2\partial_{\beta_2})f(x_0+\beta_1J_1+\beta_2J_2)\equiv0$ for all $(J_1,J_2)\in\s_{1,1}\times\s_{2,3}=\{\pm i\}\times\s^1$; this is the same as $(1,3)$-regularity by~\eqref{eq:CRwellposed};
\item $f$ is $(0,2,3)$-regular if, and only if, $(\partial_{x_0}+J_1\partial_{\beta_1}+J_2\partial_{\beta_2})f(x_0+\beta_1J_1+\beta_2J_2)\equiv0$ for all $(J_1,J_2)\in\s_{1,2}\times\s_{3,3}=(\s_\hh\cap(i\rr+j\rr))\times\{\pm k\}$; this class is the image of the class of $(0,1,3)$-regular functions when the standard basis $(1,i,j,k)$ is replaced by $(1,k,-j,i)$;
\item $f$ is $(0,1,2,3)$-regular if, and only if, $(\partial_{x_0}+J_1\partial_{\beta_1}+J_2\partial_{\beta_2}+J_3\partial_{\beta_3})f(x_0+\beta_1J_1+\beta_2J_2+\beta_3J_3)\equiv0$ for all $(J_1,J_2,J_3)\in\s_{1,1}\times\s_{2,2}\times\s_{3,3}=\{\pm i\}\times\{\pm j\}\times\{\pm k\}$; this is the same as being left Fueter-regular by~\eqref{eq:CRwellposed}.
\end{itemize}
\end{example}

Over the hypercomplex subspace $\oo$ of $\oo$, the notions of $7$-regular and $(0,7)$-regular function coincide, respectively, with the notion of \emph{octonionic monogenic} function (based on~\cite{sce}, see the recent~\cite{krausshardifferentialtopological}) and with the notion of \emph{slice-regular} function (see~\cite{rocky}). Additionally, within $\hh$, the nonstandard choice of the the $2$-fan $\rr+j\rr+k\rr$ with $\B=(1,-k,j),\B'=(1,-k,j,i)$, recovers as $2$-regular functions the theory of~\cite{moisilteodorescu}, for the reasons explained in~\cite[page 30]{perotticr}.

When a domain $\Omega$ in $\rr^{n+1}$ does not intersect the mirror $\rr_{0,t_0}$, then two different $J$-slices $\Omega_J$ never intersect and an $f\in\reg_T(\Omega,A)$ needs not be continuous, even when all restrictions $f_J$ are $\mscr{C}^2$ (whence real analytic by Remark~\ref{rmk:J-harmonic}). To single out subclasses of better-behaved $T$-regular functions, we need some further definitions. 

\begin{definition}
A domain $\Omega\subseteq\rr^{n+1}$ is called a \emph{$T$-slice domain} if it intersects the mirror $\rr_{0,t_0}$ and if, for any $J\in\torus$, the $J$-slice $\Omega_J$ is connected. For $D\subseteq\rr_{0,t_0}\times\rr^\tau$, we set
\[\Omega_D:=\{x^0+\beta_1J_1+\ldots+\beta_\tau J_\tau\in\rr^{n+1}: (x^0,\beta)\in D, J\in\torus\}\]
if $\tau\geq1$ (and $\Omega_D:=\{x^0\in\rr^{n+1} : (x^0,0)\in D\}$ if $\tau=0$, taking into account that $\rr^0=\{0\}$). A subset of $\rr^{n+1}$ is termed \emph{$T$-symmetric} if it equals $\Omega_D$ for some $D\subseteq\rr_{0,t_0}\times\rr^\tau$. The \emph{$T$-symmetric completion} $\widetilde{Y}$ of a set $Y\subseteq\rr^{n+1}$ is the smallest $T$-symmetric subset of $\rr^{n+1}$ containing $Y$. For each point $x\in\rr^{n+1}$, we denote by $\s_x$ the $T$-symmetric completion of the singleton $\{x\}$.
\end{definition}

We now define $T$-stem functions. For any $h\in\{1,\ldots,\tau\}$, consider the reflection
\[\rr^\tau\to\rr^\tau\,,\quad\beta=(\beta_1,\ldots,\beta_\tau)\mapsto\overline{\beta}^h:=(\beta_1,\ldots,\beta_{h-1},-\beta_h,\beta_{h+1},\ldots,\beta_\tau)\,.\]
For the rest of the present section, we assume $D$ to be a subset of $\rr_{0,t_0}\times\rr^\tau$, invariant under the reflection $(x^0,\beta)\mapsto(x^0,\overline{\beta}^h)$ for every $h\in\{1,\ldots,\tau\}$. Moreover, we let $\{E_K\}_{K\in\mscr{P}(\tau)}$ denote the canonical real vector basis of $\rr^{2^\tau}$: this is to avoid possible confusion with the basis of $A$ in the special case when $A$ is $C\ell(0,n)$
.

\begin{definition}
Let $F:D\to A\otimes\rr^{2^\tau}$ be a function $F=\sum_{K\in\mscr{P}(\tau)}E_KF_K$ with components $F_K:D\to A$. We say that $F$ is a \emph{$T$-stem function} if
\[F_K(x^0,\overline{\beta}^h)=\left\{
\begin{array}{ll}
F_K(x^0,\beta)&\mathrm{if\ }h\not\in K\\
-F_K(x^0,\beta)&\mathrm{if\ }h\in K
\end{array}
\right.\]
for all $K\in\mscr{P}(\tau)$, for all $(x^0,\beta)\in D$, and for all $h\in\{1,\ldots,\tau\}$.
\end{definition}

Now let us define $T$-functions.

\begin{definition}
To each $T$-stem function $F=\sum_{K\in\mscr{P}(\tau)}E_KF_K:D\to A\otimes\rr^{2^\tau}$, we associate the \emph{induced} function $f=\I(F):\Omega_D\to A$, whose value at $x=x^0+\beta_1J_1+\ldots+\beta_\tau J_\tau\in\Omega_D$ is
\[f(x):=F_{\emptyset}(x^0,\beta)+\sum_{1\leq p\leq\tau}\sum_{1\leq k_1<\ldots<k_p\leq\tau} J_{k_1}(J_{k_2}(\ldots(J_{k_{p-1}}(J_{k_p}F_{k_1\ldots k_p}(x^0,\beta)))\ldots))\,.\]
A function induced by a $T$-stem function is called a \emph{$T$-function}. We denote the class of $T$-functions $\Omega_D\to A$ by $\slice(\Omega_D,A)$. If $\Omega_D$ is a domain in $\rr^{n+1}$, then the elements of the intersection $\sr(\Omega_D,A):=\slice(\Omega_D,A)\cap\reg_T(\Omega_D,A)$ are called \emph{strongly $T$-regular} functions.
\end{definition}

\begin{remark}
The map $\I$ from the class of $T$-stem functions on $D$ to $\slice(\Omega_D,A)$ is well-defined, owing to the reflection symmetries of $T$-stem functions, and bijective.
\end{remark}

In the very special case when $\tau=0$, every subset of $\rr^{n+1}$ is $T$-symmetric, every domain $\Omega$ in $\rr^{n+1}$ is a $T$-symmetric $T$-slice domain and every function $f:\Omega\to A$ is a $T$-function, induced by a $T$-stem function $F=F_\emptyset$.

We postpone any further study of the properties of general $T$-regular functions to a forthcoming paper. In the present work, we focus henceforth on the new quaternionic function theory discovered in Example~\ref{ex:quaternions3}.


\section{$(1,3)$-regular quaternionic functions}\label{sec:(1,3)-regularity}


\subsection{Foundational material}

In this section, we work in $\hh$ with the $(1,3)$-fan $\cc\subsetneq\hh$ and the $(1,3)$-torus $\s_{2,3}=\s^1$, see Examples~\ref{ex:quaternions},~\ref{ex:quaternions2}, and~\ref{ex:quaternions3}. We decompose the variable $x=x_0+ix_1+jx_2+kx_3\in\hh$ as $x=z+\beta J$ with $z\in\cc, \beta\in\rr, J\in\s^1$. If $x\in\hh\setminus\cc$, this is true exactly when
\[z=x_0+ix_1,\quad \beta=\pm\sqrt{x_2^2+x_3^2},\quad J=\pm\frac{jx_2+kx_3}{\sqrt{x_2^2+x_3^2}}\,.\]
On the other hand, every element $z$ of the mirror $\cc\subsetneq\hh$ may be expressed as $z+0J$ for all $J\in\s^1$. Over quaternions $x\in\hh$, the $*$-involution considered is the \emph{conjugation} $x=x_0+ix_1+jx_2+kx_3\mapsto\bar x=x_0-ix_1-jx_2-kx_3$; the Euclidean norm $\Vert x\Vert=\sqrt{x\bar x}$ of $x$ is called \emph{modulus} of $x$ and denoted by the symbol $|x|$.

For $J\in\s^1$, the spaces $\rr^{2+1}_J:=\Span_\rr(1,i,J)$ are the hyperplanes of $\hh$ through $\cc$. We have $\rr^{2+1}_J\cap\rr^{2+1}_K=\cc$ when $K\neq\pm J$ and $\rr^{2+1}_J=\rr^{2+1}_{-J}$. When $J=j$, the hyperplane $\rr^{2+1}_j$ is the standard $\rr^{2+1}$ used in the theory of \emph{monogenic} functions (also called \emph{$C\ell(0,2)$ left-holomorphic functions} in~\cite{librogurlebeck2}): namely, quaternion-valued $\mscr{C}^1$ functions of the variable $x_0+ix_1+jx_2$ belonging to the kernel of the operator $\partial_{x_0}+i\partial_{x_1}+j\partial_{x_2}$. For any $J\in\s^1$, $(i,J,iJ)$ is a Hamiltonian triple (see~\cite[\S8.1]{ebbinghaus}) and $\psi_J:\rr^{2+1}\to\rr^{2+1}_J$, $\psi_J(x_0+ix_1+jx_2)=x_0+ix_1+Jx_2$ is the restriction of a $*$-isomorphism $\hh\to\hh$. Thus, if $V\subseteq\rr^{2+1}_J$ is open, then the right $\hh$-module $\mon_J(V,\hh)$ of $J$-monogenic functions has the same properties as a class of monogenic functions over $C\ell(0,2)$. In particular, $J$-monogenicity implies real analyticity, see Remark~\ref{rmk:seriesexpansionofJmonogenic}, and is preserved under composition with translations in $\rr^{2+1}_J$.

\begin{remark}\label{rmk:complextranslation}
Given a domain $\Omega$ in $\hh$, the class $\reg_{(1,3)}(\Omega,\hh)$ of all $(1,3)$-regular functions $\Omega\to\hh$ is a right $\hh$-module. Moreover, if $f\in\reg_{(1,3)}(\Omega,\hh)$ and $z_0\in\Omega\cap\cc$, then setting $g(x):=f(x+z_0)$ defines a $g\in\reg_{(1,3)}(\Omega-z_0,\hh)$. This $g$ is $(1,3)$-slice preserving if, and only if, $f$ is.
\end{remark}

We remark that $(1,3)$-symmetry is circular symmetry with respect to the plane $\cc$ within $\hh$. 
We fix henceforth $D\subseteq\cc\times\rr$, invariant under the reflection $(z,\beta)\mapsto(z,-\beta)$. A function $F=F_\emptyset+E_1F_1:D\to\hh\otimes\rr^2$ is a \emph{$(1,3)$-stem function} if $F_\emptyset,F_1$ are, respectively, even and odd in $\beta$. The $(1,3)$-function $f=\I(F)$, defined as $f(z+\beta J):=F_\emptyset(z,\beta)+JF_1(z,\beta)$ for $(z,\beta)\in D,J\in\s^1$, is affine along each circle $\s_x$.


\subsection{Polynomial functions}

For all $k\in\nn$, let $U_k$ denote the right $\hh$-submodule of $\reg_{(1,3)}(\hh,\hh)$ consisting of those elements $f$ such that $f(x_0+ix_1+jx_2+kx_3)$ is a $k$-homogeneous polynomial map in the four real variables $x_0,x_1,x_2,x_3$. We construct a right $\hh$-basis $\F_k$ of $U_k$, as follows.

\begin{definition}
For $\k=(k_1,k_2)\in\zz^2$, we set $|\k|:=k_1+k_2$. If $k_1<0$ or $k_2<0$, we define $\T_\k(x):\equiv0$; we define $\T_{(0,0)}:\equiv1$; if $\k\in\nn^2\setminus\{(0,0)\}$, we define $\T_\k$ by the equality
\[|\k|\T_\k(x):=k_1\T_{(k_1-1,k_2)}(x)(x_1-(-1)^{k_2}ix_0)+k_2\T_{(k_1,k_2-1)}(x)(x_0+jx_2+kx_3)\,.\]
For all $k\in\nn$, we define $\F_k:=\{\T_\k\}_{|\k|=k}$.
\end{definition}

\begin{example}\label{ex:lowerdegreepolynomials}
$\F_0=\{\T_{(0,0)}\}=\{1\}$, while $\F_1,\F_2$ consist of the functions
\begin{align*}
&\T_{(1,0)}(x)=x_1-ix_0\,,\quad\T_{(0,1)}(x)=x_0+jx_2+kx_3\,,\\
&\T_{(2,0)}(x)=(x_1-ix_0)^2\,,\quad\T_{(0,2)}(x)=(x_0+jx_2+kx_3)^2\,,\\
&2\T_{(1,1)}(x)=(x_0+jx_2+kx_3)(x_1+ix_0)+(x_1-ix_0)(x_0+jx_2+kx_3)\,.
\end{align*}
None of the elements of $\F_1,\F_2$ is a slice-regular function $\hh\to\hh$. Moreover, $\T_{(0,1)},\T_{(0,2)},\T_{(1,1)}:\hh\to\hh$ are not Fueter-regular. 
\end{example}

Clearly, each $\T_\k$ is a $|\k|$-homogeneous polynomial map in the four real variables $x_0,x_1,x_2,x_3$. To prove that $\F_k$ is, indeed, a basis of $U_k$, we rely on some well-known properties of monogenic functions, described in~\cite[\S6.1.2]{librogurlebeck2}.

\begin{remark}
For any $J\in\s^1,k\in\nn$, let $U^J_k$ denote the right submodule of $\mon_J(\rr^{2+1}_J,\hh)$ consisting of those elements $P$ such that $P(x_0+ix_1+\beta J)$ is a $k$-homogeneous polynomial map in the real variables $x_0,x_1,\beta$. Then $U^J_0$ is spanned by $\P^J_{(0,0)}:\equiv1$, while $U^J_1$ is spanned by $\P^J_{(1,0)}:=x_1-ix_0$, also called the Fueter variable $\zeta_1$, and by $\P^J_{(0,1)}:=\beta-Jx_0$, also called the Fueter variable $\zeta_{2,J}$. If we further define $\P^J_\k:\equiv0$ for $\k\in\zz^2$ with $k_1<0$ or $k_2<0$ and, for all remaining $\k\in\nn^2$,
\begin{align*}
|\k|\P^J_\k&:=k_1\P^J_{(k_1-1,k_2)}\zeta_1+k_2\P^J_{(k_1,k_2-1)}\zeta_{2,J}=k_1\zeta_1\P^J_{(k_1-1,k_2)}+k_2\zeta_{2,J}\P^J_{(k_1,k_2-1)}\,,
\end{align*}
then each $\P^J_\k$ preserves $\rr^{2+1}_J$ and $\{\P^J_\k\}_{|\k|=k}$ is a basis of $U^J_k$. To be more precise, we need some preparation. For every open $V\subseteq\rr^{2+1}_J$ and every $\h=(h_0,h_1,h_2)\in\nn^3$, we define $|\h|:=h_0+h_1+h_2$ and the operator $\nabla^\h_J:\mscr{C}^{|\h|}(V,\hh)\to\mscr{C}^0(V,\hh)$ by
\[\nabla^\h_J\phi(x_0+ix_1+\beta J):=\partial^{h_0}_0\partial^{h_1}_1\partial^{h_2}_{2,J}\phi(x_0+ix_1+\beta J),\quad\partial_0:=\partial_{x_0},\partial_1:=\partial_{x_1},\partial_{2,J}:=\partial_\beta\,.\]
More precisely, $\left((\nabla^\h_J\phi)\circ\psi_J\right)(x_0+ix_1+jx_2):=\left(\left(\partial_{x_0}^{h_0}\partial_{x_1}^{h_1}\partial_{x_2}^{h_2}\right)(\phi\circ\psi_J)\right)(x_0+ix_1+jx_2)$. Since the operator $\nabla^\h_J$ commutes with $\debar_J$ on $\mscr{C}^{|\h|+1}(V,\hh)$, it maps $J$-monogenic functions into $J$-monogenic functions. With this construction in mind, we see that $\{\P^J_\k\}_{|\k|=k}$ is a basis of $U^J_k$ because every $P\in U^J_k$ can be expressed as
\begin{equation}\label{eq:Jpolynomialexpansion}
P(x_0+ix_1+\beta J)=\sum_{|\k|=k}\P^J_\k(x_0+ix_1+\beta J)a_\k,\quad a_\k:=\frac{1}{\k!}\nabla^{(0,\k)}_JP(0)
\end{equation}
for $x_0+ix_1+\beta J\in\rr^{2+1}_J$. Moreover, $P$ preserves $\rr^{2+1}_J$ under the following additional assumptions: $a_\k\in\rr^{2+1}_J$ when $k_1=0=k_2$; $a_\k\in\cc_J$ when $k_1=0\neq k_2$, $a_\k\in\cc$ when $k_1\neq0=k_2$, $a_\k\in\rr$ when $k_1\neq0\neq k_2$.
\end{remark}

We also need the next definition and lemma, where $\Omega$ is a fixed domain in $\hh$, $J\in\s^1$ is fixed and $V$ is a fixed open subset of $\rr^{2+1}_J$.

\begin{definition}
For every $\h=(h_0,h_1,h_2)\in\nn^3$, we define an operator $\delta^{\h}:\mscr{C}^{|\h|}(\Omega,\hh)\to\mscr{C}^{0}(\Omega,\hh)$, as follows. For $t\in\nn$, we set
\begin{align*}
&\delta^{(h_0,h_1,2t)}:=\partial^{h_0}_0\partial^{h_1}_1(\partial_0-i\partial_1)^{t}(\partial_0+i\partial_1)^{t}=\partial^{h_0}_0\partial^{h_1}_1(\partial_0^2+\partial_1^2)^{t}\,,\\
&\delta^{(h_0,h_1,2t+1)}:=\partial^{h_0}_0\partial^{h_1}_1(\partial_0-i\partial_1)^{t}(\partial_0+i\partial_1)^{t+1}=\partial^{h_0}_0\partial^{h_1}_1(\partial_0^2+\partial_1^2)^{t}(\partial_0+i\partial_1)\,.
\end{align*}
The same formulas define operators $\delta^{(h_0,h_1,2t)}_J,\delta^{(h_0,h_1,2t+1)}_J$ on $\mscr{C}^{|\h|}(V,\hh)$.
\end{definition}

\begin{lemma}\label{lem:PTnabladelta}
Fix $\k=(k_1,k_2)\in\zz^2,\h=(h_0,h_1,h_2)\in\nn^3$. The restriction $(\T_\k)_J$ is the $J$-monogenic polynomial $\P^J_\k J^{k_2}\in U^J_k$. In particular, $(\T_\k)_J$ maps $\rr^{2+1}_J$ into itself when $k_2$ is even; into $\rr^{2+1}_JJ=\rr+j\rr+k\rr$ when $k_2$ is odd. If $\phi\in\mon_J(V,\hh)$, then $\delta^{\h}_J\phi=J^{-h_2}\nabla^\h_J\phi$. If $f\in\reg_{(1,3)}(\Omega,\hh)$ is $\mscr{C}^{|\h|}$, then $(\delta^{\h}f)_J=J^{-h_2}\nabla^\h_Jf_J$.
\end{lemma}

\begin{proof}
We first prove that $(\T_\k)_J=\P^J_\k J^{k_2}$, by induction on $|\k|$. The thesis is obviously true when $k_1<0$ or $k_2<0$, since in this case $\T_\k$ and $\P^J_\k$ both vanish identically. Similarly, $(\T_{(0,0)})_J=\P^J_{(0,0)} J^0$ since $\T_{(0,0)}:\equiv1\equiv:\P^J_{(0,0)}$ Assume the thesis true for all $\k'$ with $|\k'|=k-1$. We can prove it for $\k$ with $|\k|=k$ by means of the following computation, where we omit the variable $x_0+ix_1+\beta J$:
\begin{align*}
k(\T_\k)_J&=k_1\T_{(k_1-1,k_2)}\cdot(x_1-(-1)^{k_2}ix_0)+k_2\T_{(k_1,k_2-1)}\cdot(x_0+\beta J)\\
&=k_1\P^J_{(k_1-1,k_2)}J^{k_2}(x_1-(-1)^{k_2}ix_0)+k_2\P^J_{(k_1,k_2-1)}J^{k_2-1}(x_0+\beta J)\\
&=\left(k_1\P^J_{(k_1-1,k_2)}\cdot(x_1-ix_0)+k_2\P^J_{(k_1,k_2-1)}\cdot(\beta-Jx_0)\right)J^{k_2}\\
&=k\P^J_\k J^{k_2}\,.
\end{align*}

Now we fix $\phi\in\mon_J(V,\hh)$ and prove $\nabla^{\h}_J\phi=J^{h_2}\delta^{\h}_J\phi$ by induction on $h_2\in\nn$.
\begin{align*}
&\nabla^{(h_0,h_1,0)}_J\phi=\partial^{h_0}_0\partial^{h_1}_1\phi=\delta^{(h_0,h_1,0)}_J\phi=J^{0}\delta^{(h_0,h_1,0)}_J\phi\,,\\
&\nabla^{(h_0,h_1,1)}_J\phi=\partial^{h_0}_0\partial^{h_1}_1\partial_{2,J}\phi=\partial^{h_0}_0\partial^{h_1}_1J(\partial_0+i\partial_1)\phi=J\delta^{(h_0,h_1,1)}_J\phi\,.
\end{align*}
In the second chain of equalities, we have used the fact that $\partial_{2,J}\phi=J(\partial_0+i\partial_1)\phi$, which is a consequence of the equality $0\equiv\debar_J \phi=(\partial_0+i\partial_1+J\partial_{2,J})\phi$. We remark that, for all $\h\in\nn^3$, the function $\nabla^{\h}_J\phi$ is still $J$-monogenic, whence the equality $\partial_{2,J}\nabla^{\h}_J\phi=J(\partial_0+i\partial_1)\nabla^{\h}_J\phi$ follows. We are now ready for the induction step. Under the inductive hypothesis $\nabla^{(h_0,h_1,2t-1)}_J\phi=J^{2t-1}\delta^{(h_0,h_1,2t-1)}_J\phi$, we have
\begin{align*}
&\nabla^{(h_0,h_1,2t)}_J\phi=\partial_{2,J}\nabla^{(h_0,h_1,2t-1)}_J\phi=J(\partial_0+i\partial_1)\nabla^{(h_0,h_1,2t-1)}_J\phi\\
&=J(\partial_0+i\partial_1)J^{2t-1}\delta^{(h_0,h_1,2t-1)}_J\phi=J^{2t}(\partial_0-i\partial_1)\delta^{(h_0,h_1,2t-1)}_J\phi=J^{2t}\delta^{(h_0,h_1,2t)}_J\phi\,,\\
&\nabla^{(h_0,h_1,2t+1)}_J\phi=\partial_{2,J}\nabla^{(h_0,h_1,2t)}_J\phi=J(\partial_0+i\partial_1)\nabla^{(h_0,h_1,2t)}_J\phi\\
&=J(\partial_0+i\partial_1)J^{2t}\delta^{(h_0,h_1,2t)}_J\phi=J^{2t+1}(\partial_0+i\partial_1)\delta^{(h_0,h_1,2t)}_J\phi=J^{2t+1}\delta^{(h_0,h_1,2t+1)}_J\phi\,.
\end{align*}
This completes our induction step, whence our proof of the equality $\nabla^{\h}_J\phi=J^{h_2}\delta^{\h}_J\phi$.

Finally, if $f\in\reg_{(1,3)}(\Omega,\hh)$ is $\mscr{C}^{|\h|}$, then the definitions of $\delta^{\h},\delta^{\h}_J$ and the equality established for $J$-monogenic functions yield the desired equality $(\delta^{\h}f)_J=\delta^{\h}_Jf_J=J^{-h_2}\nabla^\h_Jf_J$.
\end{proof}

We are now ready to prove the announced result.

\begin{theorem}\label{thm:polynomialexpansion}
The family $\F_k$ is a basis for $U_k$. Namely, for every $P\in U_k$,
\begin{equation}\label{eq:polynomialexpansion}
P(x)=\sum_{|\k|=k}\T_\k(x)c_\k,\quad c_\k:=\frac{1}{\k!}\delta^{(0,\k)}P(0)\,.
\end{equation}
Moreover, $P$ is $(1,3)$-slice preserving under the following assumptions on $c_\k$, valid for all $\k=(k_1,k_2)$ with $|\k|=k$: $c_\k\in\cc$ when $k_2=0$; $c_\k\in\rr$ when either $k_1=0\neq k_2$ or $k_1\neq0\neq k_2\in2\nn$; $c_\k=0$ when $k_1\neq0, k_2\in2\nn+1$.
\end{theorem}

\begin{proof}
Fix $k\in \nn$. We take several steps, applying Lemma~\ref{lem:PTnabladelta} at each step.

We first prove the inclusion $\F_k\subseteq U_k$. We know that each function $\T_\k$ is a $|\k|$-homogenous polynomial. Since, for all $J\in\s^1$, the restriction $(\T_\k)_J=\P^J_\k J^{k_2}$ is $J$-monogenic, we also have $\T_\k\in\reg_{(1,3)}(\hh,\hh)$. The desired inclusion follows.

We now prove that the elements of $\F_k$ are linearly independent. For $\{c_\k\}_{|\k|=k}\subset\hh$, if $P(x):=\sum_{|\k|=k}\T_\k(x) c_k$ vanishes identically in $\hh$, then $P_J=\sum_{|\k|=k}\P^J_\k J^{k_2} c_\k$ vanishes identically in $\rr^{2+1}_J$. Since $\{\P^J_\k\}_{|\k|=k}$ is a basis of $U^J_k$, it follows that $J^{k_2}c_\k=0$ (whence $c_\k=0$) for all $\k\in\nn^2$ with $|\k|=k$.

We now prove that formula~\eqref{eq:polynomialexpansion} is true for all $P\in U_k$, whence the family $\F_k$ spans $U_k$. It suffices to prove that the polynomial function $\widetilde{P}:=\sum_{|\k|=k}\T_\k c_\k$ coincides with $P$. This is true, because for all $J\in\s^1$
\[\widetilde{P}_J=\sum_{|\k|=k}(\T_\k)_J c_\k=\sum_{|\k|=k}\P^J_\k J^{k_2} c_\k=\sum_{|\k|=k}\P^J_\k\frac{1}{\k!}\nabla^{(0,\k)}_JP_J(0)=P_J\,.\]

The final statement follows from the analogous property of $U^J_k$, taking into account that $\cc=\bigcap_{J\in\s^1}\rr^{2+1}_J$, that $\rr=\bigcap_{J\in\s^1}\cc_J$ and that $\{0\}=\bigcap_{J\in\s^1}J\rr$.
\end{proof}

Taking into account Example~\ref{ex:lowerdegreepolynomials}, we draw the following conclusion.

\begin{corollary}
For any domain $\Omega$ in $\hh$, the function space $\reg_{(1,3)}(\Omega,\hh)$ is distinct both from the space $\reg_{3}(\Omega,\hh)$ of left Fueter-regular functions and from the space $\reg_{(0,3)}(\Omega,\hh)$ of slice-regular quaternionic functions.
\end{corollary}

We point out that we could not have used the $J$-monogenic polynomials $\P^J_\k$ to construct a well-defined $(1,3)$-regular function because  $\P^{-J}_\k(x_0+ix_1-\beta(-J))=(-1)^{k_2}\P^J_\k(x_0+ix_1+\beta J)$. Nor could we have used $\nabla^\h_J$ to construct an operator on $(1,3)$-regular functions, because $(\nabla^\h_{-J}\phi)(x_0+ix_1-\beta(-J))=(-1)^{h_2}(\nabla^\h_J\phi)(x_0+ix_1+\beta J)$. We now prove that the $\T_\k$'s are strongly $(1,3)$-regular.

\begin{proposition}\label{prop:AkandBk}
For all $\k\in\zz^2$, we have $\T_\k\in\sr(\hh,\hh)$. Indeed, there exist (real) polynomial functions $A_\k,B_\k:\cc\times\rr\to\cc$ such that
\[\T_\k(x)=A_\k\left(x_0+ix_1,x_2^2+x_3^3\right)+(jx_2+kx_3)B_\k\left(x_0+ix_1,x_2^2+x_3^3\right)\]
for all $x=x_0+ix_1+jx_2+kx_3\in\hh$. Moreover, both $(x_0,x_1,\beta)\mapsto A_\k(x_0+ix_1,\beta^2)$ and $(x_0,x_1,\beta)\mapsto\beta B_\k(x_0+ix_1,\beta^2)$) are $|\k|$-homogeneous. As a consequence: $\T_\k$ preserves $\cc$; the modulus $|\T_\k(x_0+ix_1+jx_2+kx_3)|$ depends on $\k,x_0,x_1,x_2^2+x_3^3$ but is independent of the imaginary unit $\frac{jx_2+kx_3}{\sqrt{x_2^2+x_3^3}}\in\s^1$.
\end{proposition}

\begin{proof}
If suffices to set $A_\k:\equiv0\equiv:B_\k$ for all $\k=(k_1,k_2)\in\zz^2$ with $k_1<0$ or $k_2<0$, as well as $A_{(0,0)}:\equiv1,B_{(0,0)}:\equiv0$ and, for $\k\in\nn^2\setminus\{(0,0)\}$,
\begin{align*}
|\k|A_\k(x_0+ix_1,\gamma)&:=k_1A_{(k_1-1,k_2)}(x_1-(-1)^{k_2}ix_0)+k_2A_{(k_1,k_2-1)}x_0-k_2\overline{B_{(k_1,k_2-1)}}\gamma\,,\\
|\k|B_\k(x_0+ix_1,\gamma)&:=k_1B_{(k_1-1,k_2)}(x_1-(-1)^{k_2}ix_0)+k_2\overline{A_{(k_1,k_2-1)}}+k_2B_{(k_1,k_2-1)}x_0\,.\qedhere
\end{align*}
\end{proof}


\subsection{Integral and series representation, identity principle}

In this subsection, our first aim is providing an integral representation formula for $(1,3)$-regular quaternionic functions. We need the notation $B_J(y_0,R):=B(y_0,R)\cap\rr^{2+1}_J$, valid only for $y_0\in\rr^{2+1}_J$ and for $R>0$. We begin with Cauchy's integral formula for $J$-monogenic functions, see~\cite[Theorem 7.12]{librogurlebeck2} (which also specifies the hypersurface integrals considered here). The subsequent proposition follows at once.

\begin{remark}
Let $J\in\s^1,y_0\in\rr^{2+1}_J, R>0$ and $B_J:=B_J(y_0,R)$. If $\phi\in\mon_J(V,\hh)$ for some open neighborhood $V$ of the closure of $B_J$ in $\rr^{2+1}_J$, then for all $x_0+ix_1+\beta J\in B_J$
\begin{equation}\label{eq:cauchyintegralJmonogenic}
\phi(x_0+ix_1+\beta J)=\frac1{4\pi}\int_{\partial B_J}\frac{\overline{y-x_0-ix_1-\beta J}}{|y-x_0-ix_1-\beta J|^{2+1}}dy^*\phi(y)\,.
\end{equation}
\end{remark}

\begin{proposition}
Let $\Omega$ be a domain in $\hh$ and let $f:\Omega\to\hh$ be a $(1,3)$-regular function. If $J\in\s^1,y_0\in\Omega_J,R>0$ are such that $\Omega$ contains the closure of $B_J:=B_J(y_0,R)$, then for all $x\in B_J$
\begin{equation}\label{eq:cauchyintegral}
f(x)=\frac1{4\pi}\int_{\partial B_J}\frac{\overline{y-x}}{|y-x|^{2+1}}dy^*f_J(y)\,.
\end{equation}
\end{proposition}

Our next aim is expanding $(1,3)$-regular quaternionic functions into polynomial series at every point of the mirror $\cc$. To do so, we rely on a well-known property of monogenic functions, described in~\cite[Theorem 9.24]{librogurlebeck2}:

\begin{remark}\label{rmk:seriesexpansionofJmonogenic}
Let $J\in\s^1,R>0,B_J:=B_J(0,R),\phi\in\mon_J(B_J,\hh)$. Then 
\begin{equation}\label{eq:seriesexpansionofJmonogenic}
\phi=\sum_{k\in\nn}\sum_{|\k|=k}\P^J_\k\frac{1}{\k!}\nabla^{(0,\k)}_J\phi(0)\,.
\end{equation}
Moreover, the series in formula~\eqref{eq:seriesexpansionofJmonogenic} converges absolutely and uniformly on compact sets in $B_J$. Now, fix a connected open subset $V$ of $\rr^{2+1}_J$ and $\phi\in\mon_J(V,\hh)$. Then $\phi$ is real analytic because for every $B_J(y_0,R)\subseteq V$ formula~\eqref{eq:seriesexpansionofJmonogenic} applies to $y\mapsto\phi_{|_{B_J(y_0,R)}}(y+y_0)$. In particular: if $\phi'\in\mon_J(V,\hh)$ coincides with $\phi$ on a subset of $V$ whose interior is not empty (or, more generally, whose Hausdorff dimension is greater than $2$), then $\phi=\phi'$ throughout $V$.
\end{remark}

\begin{theorem}
Let $\Omega$ be a domain in $\hh$, including the open ball $B(z_0,R)$ for some $z_0\in\cc$ and some $R>0$. If $f:\Omega\to\hh$ is a $(1,3)$-regular function, then
\begin{equation}\label{eq:seriesexpansion}
f(x)=\sum_{k\in\nn}\sum_{|\k|=k}\T_\k(x-z_0)c_\k,\quad c_\k:=\frac{1}{\k!}\delta^{(0,\k)}f(z_0)
\end{equation}
for all $x\in B(z_0,R)$. Moreover, the series on the right-hand side of formula~\eqref{eq:seriesexpansion} converges absolutely and uniformly on compact sets in $B(z_0,R)$. Finally, $f$ is $(1,3)$-slice preserving if we make the following assumptions: $c_\k\in\cc$ when $k_2=0$; $c_\k\in\rr$ when either $k_1=0\neq k_2$ or $k_1\neq0, k_2\in2\nn$; $c_\k=0$ when $k_1\neq0, k_2\in2\nn+1$.
\end{theorem}

\begin{proof}
We assume $z_0=0$, without loss of generality because Remark~\ref{rmk:complextranslation} allows us to precompose our $(1,3)$-regular function $f$ with the translation $x\mapsto x+z_0$. Let us consider on $B:=B(0,R)$ the series
\[\widetilde{f}(x):=\sum_{k\in\nn}\sum_{|\k|=k}\T_\k(x)c_\k,\quad c_\k:=\frac{1}{\k!}\delta^{(0,\k)}f(0)\,.\]
For every compact subset $C$ of $B$, its symmetric completion $\widetilde{C}$ is also a compact subset of $B$. If we fix $J_0\in\s^1$ and take any $x=x_0+ix_1+\beta J\in \widetilde{C}$, then
\[\left\vert\T_\k(x) c_\k\right\vert=\left\vert\T_\k(x)\right\vert\vert c_\k\vert=\left\vert\T_\k(x_0+ix_1+\beta J_0)\right\vert\vert c_\k\vert
=\left\vert\P^{J_0}_\k(x_0+ix_1+\beta J_0)\right\vert\left\vert\frac{1}{\k!}\nabla^{(0,\k)}_{J_0}f_{J_0}(0)\right\vert\]
for all $\k\in\nn^2$. For the second and third equalities, we have used Proposition~\ref{prop:AkandBk} and Lemma~\ref{lem:PTnabladelta}, respectively. The series $\sum_{k\in\nn}\sum_{|\k|=k}\P^{J_0}_\k(x_0+ix_1+\beta J_0)\frac{1}{\k!}\nabla^{(0,\k)}_{J_0}f_{J_0}(0)$ is the expansion~\eqref{eq:seriesexpansionofJmonogenic} of $f_{J_0}$: this expansion converges absolutely and uniformly to $f_{J_0}$ on each compact subset of $B_{J_0}$, including the $J_0$-slice $(\widetilde{C})_{J_0}$ of $\widetilde{C}$. It follows at once that the series $\widetilde{f}(x)$ converges absolutely and uniformly for $x\in C$. Its sum defines a function $\widetilde{f}:B\to\hh$. For every $J\in\s^1$,
\begin{align*}
\widetilde{f}_J=\sum_{k\in\nn}\sum_{|\k|=k}(\T_\k)_Jc_\k=\sum_{k\in\nn}\sum_{|\k|=k}\P^J_\k\frac{1}{\k!}\nabla^{(0,\k)}_Jf_J(0)=(f_J)_{|_{B_J}}\,,
\end{align*}
thanks to Lemma~\ref{lem:PTnabladelta} and formula~\eqref{eq:seriesexpansionofJmonogenic}. Thus, $\widetilde{f}$ coincides with $f_{|_B}$ and formula~\eqref{eq:seriesexpansion} is proven. The final statement follows from its analog in Theorem~\ref{thm:polynomialexpansion}.
\end{proof}

Besides its independent interest, formula~\eqref{eq:seriesexpansion} serves to prove the following principle, valid only on $(1,3)$-slice domains.

\begin{theorem}[Identity Principle]\label{thm:identityprinciple}
Let $\Omega$ be a $(1,3)$-slice domain in $\hh$ and let $f,g:\Omega\to\hh$ be $(1,3)$-regular functions. If there exist $J_0\in\s^1$ and subset of $\Omega_{J_0}$ whose interior is not empty (or whose Hausdorff dimension is greater than $2$) where $f$ and $g$ coincide, then $f=g$ throughout $\Omega$.
\end{theorem}

\begin{proof}
Since $f,g$ are $(1,3)$-regular functions on $\Omega$, for all $J\in\s^1$ we have $f_J,g_J\in\mon_J(\Omega_J,\hh)$ and we can apply Remark~\ref{rmk:seriesexpansionofJmonogenic}. Our hypotheses yield that $f_{J_0}$ and $g_{J_0}$ coincide in their connected domain $\Omega_{J_0}$. Since $\Omega$ is a slice domain, there exist $z_0\in\cc, R>0$ such that the open ball $B:=B(z_0,R)$ is contained in $\Omega$. We have
\[\delta^{(0,\k)}f(z_0)=J_0^{-k_2}\nabla^{(0,\k)}_{J_0}f_{J_0}(z_0)=J_0^{-k_2}\nabla^{(0,\k)}_{J_0}g_{J_0}(z_0)=\delta^{(0,\k)}g(z_0)\,.\]
This information, when performing expansion~\eqref{eq:seriesexpansion} at $z_0$ for both $f$ and $g$, yields $f_{|_B}=g_{|_B}$. In particular, for every $J\in\s^1$ the real analytic functions $f_J$ and $g_J$ coincide in $B_J$, whence in their connected domain $\Omega_J$. We conclude that $f=g$ throughout $\Omega$, as announced.
\end{proof}


\subsection{Representation formula}

We now aim at proving that if a $(1,3)$-symmetric set $\Omega:=\Omega_D$ is a $(1,3)$-slice domain then $\reg_{(1,3)}(\Omega,\hh)=\sr(\Omega,\hh)$, i.e., all $(1,3)$-regular functions on $\Omega$ are strongly $(1,3)$-regular. To this end, we establish the next property.

\begin{theorem}[General Representation Formula]
Assume $\Omega:=\Omega_D$ to be a $(1,3)$-slice domain and let $f\in\reg_{(1,3)}(\Omega,\hh)$. For all $I,J,K\in\s^1$ with $J\neq K$ and for all $z+\beta I\in\Omega$,
\begin{align}\label{eq:generalrepresentationformula}
f(z+\beta I)&=(J-K)^{-1}(Jf(z+\beta J)-Kf(z+\beta K))+I(J-K)^{-1}(f(z+\beta J)-f(z+\beta K))\notag\\
&=\left((J-K)^{-1}J+I(J-K)^{-1}\right)f(z+\beta J)-\left((J-K)^{-1}K+I(J-K)^{-1}\right)f(z+\beta K)\,.
\end{align}
In particular,
\begin{align}\label{eq:representationformula}
f(z+\beta I)&=\frac{1}{2}(f(z+\beta J)+f(z-\beta J))+I\frac{J}{2}(f(z-\beta J)-f(z+\beta J))\\
&=\frac{1-IJ}{2}f(z+\beta J)+\frac{1+IJ}{2}f(z-\beta J)\,.\notag
\end{align}
\end{theorem}

\begin{proof}
To prove formula~\eqref{eq:representationformula}, we show that the function $g:\Omega\to\hh$ defined as
\begin{align*}
g(z+\beta I)&:=\frac{1-IJ}{2}f_J(z+\beta J)+\frac{1+IJ}{2}f_{-J}(z-\beta J)
\end{align*}
(for $I\in\s^1,z+\beta I\in\Omega$) coincides with $f$. A long computation proves that
\[(\debar_Ig_I)(z+\beta I)=\frac{1-IJ}{2}(\debar_Jf_J)(z+\beta J)+\frac{1+IJ}{2}(\debar_{-J}f_{-J})(z-\beta J)\,.\]
This expression vanishes identically because $f_J=f_{-J}$ is both $J$-monogenic and $(-J)$-monogenic. A direct computation for $I=J$ shows that $g_J=f_J$. By the Identity Principle~\ref{thm:identityprinciple}, $f$ and $g$ coincide throughout $\Omega$.

With formula~\eqref{eq:representationformula} available, we can prove formula~\eqref{eq:generalrepresentationformula} by direct computation. We start with the last expression in formula~\eqref{eq:generalrepresentationformula} and substitute $\frac{1-KJ}{2}f(z+\beta J)+\frac{1+KJ}{2}f(z-\beta J)$ for $f(z+\beta K)$. A long but straightforward computation shows that 
the last expression in formula~\eqref{eq:generalrepresentationformula} equals $\frac{1-IJ}2f(z+\beta J)+\frac{1+IJ}2f(z-\beta J)$, which in turn equals $f(z+\beta I)$ by formula~\eqref{eq:representationformula}.
\end{proof}

We are now ready for the announced result.

\begin{theorem}\label{thm:preslicefunction}
If $\Omega:=\Omega_D$ is a $(1,3)$-slice domain, $\reg_{(1,3)}(\Omega,\hh)=\sr(\Omega,\hh)$. Namely, every $f\in\reg_{(1,3)}(\Omega,\hh)$ is a $(1,3)$-function, i.e, there exist unique real analytic functions $F_\emptyset,F_1:D\to\hh$, respectively even and odd in $\beta$, such that
\begin{equation}\label{eq:preslicefunction}
f(z+\beta J)=F_\emptyset(z,\beta)+JF_1(z,\beta)
\end{equation}
for all $J\in\s^1$ and all $z+\beta J\in\Omega$. In particular, $f$ is real analytic. Moreover,
\begin{align}\label{eq:regularityofprestem}
&(\partial_{x_0}+i\partial_{x_1})F_\emptyset-\partial_\beta F_1=0\,,\\
&\partial_\beta F_\emptyset+(\partial_{x_0}-i\partial_{x_1}) F_1=0\,,\notag
\end{align}
whence $F_\emptyset(x_0+ix_1,\beta),F_1(x_0+ix_1,\beta)$ are harmonic in $x_0,x_1,\beta$. Finally, $f$ is $(1,3)$-slice preserving if, and only if, $F_\emptyset(D)\subseteq\cc$ and $F_1(D)\subseteq\rr+j\rr+k\rr$.
\end{theorem}

\begin{proof}
If a stem function $F:=F_\emptyset+E_1F_1:D\otimes\rr^2\to\hh$ such that $f=\I(F)$ exists, then it is unique because $\I$ is bijective. Let us prove existence. Fix $J\in\s^1$: since $f_J$ is real analytic, setting
\begin{align}\label{eq:F1andF2}
F_\emptyset(z,\beta)&:=\frac{1}{2}(f_J(z+\beta J)+f_J(z-\beta J))\\
F_1(z,\beta)&:=\frac{J}{2}(f_J(z-\beta J)-f_J(z+\beta J))\notag
\end{align}
defines two real analytic functions $F_\emptyset,F_1:D\to\hh$, with $F_\emptyset(z,\beta)$ even in $\beta$ and $F_1(z,\beta)$ odd in $\beta$. By the Representation Formula~\eqref{eq:representationformula}, equality~\eqref{eq:preslicefunction} holds. Now,
\begin{align*}
0&\equiv\debar_Jf_J(z+\beta J)\\
&=(\partial_{x_0}+i\partial_{x_1}+J\partial_\beta)F_\emptyset+(\partial_{x_0}+i\partial_{x_1}+J\partial_\beta)(JF_1)\\
&=(\partial_{x_0}+i\partial_{x_1})F_\emptyset-\partial_\beta F_1+J(\partial_\beta F_\emptyset+(\partial_{x_0}-i\partial_{x_1}) F_1)\,.
\end{align*}
(with $J\in\s^1$ arbitrary) implies equalities~\eqref{eq:regularityofprestem}. For $\Delta=\partial_{x_0}^2+\partial_{x_1}^2+\partial_\beta^2$,
\begin{align*}
\Delta F_\emptyset&=(\partial_{x_0}-i\partial_{x_1})(\partial_{x_0}+i\partial_{x_1})F_\emptyset+\partial_\beta^2F_\emptyset=(\partial_{x_0}-i\partial_{x_1})\partial_\beta F_1-\partial_\beta(\partial_{x_0}-i\partial_{x_1})F_1\,,\\
\Delta F_1&=(\partial_{x_0}+i\partial_{x_1})(\partial_{x_0}-i\partial_{x_1})F_1+\partial_\beta^2F_1=-(\partial_{x_0}+i\partial_{x_1})\partial_\beta F_\emptyset+\partial_\beta(\partial_{x_0}+i\partial_{x_1}) F_\emptyset
\end{align*}both vanish identically.
Finally, taking into account formula~\eqref{eq:F1andF2}, we see that $f$ is $(1,3)$-slice preserving if, and only if, $F_\emptyset$ takes values in $\bigcap_{J\in\s^1}\rr^{2+1}_J=\cc$ and $F_1$ takes values in $\bigcap_{J\in\s^1}J\rr^{2+1}_J=\rr+j\rr+k\rr$.
\end{proof}

We conclude with a relevant example and with a more general remark.

\begin{example}
Let $z_0\in\cc, R>0,B:=B(z_0,R),\{c_\k\}_{\k\in\nn^2}\subset\hh$ be such that
\[f(x)=\sum_{k\in\nn}\sum_{|\k|=k}\T_\k(x-z_0)c_\k\]
converges absolutely and uniformly in $B$, defining a $(1,3)$-regular function $f:B\to\hh$. Referring to the polynomial functions $A_\k,B_\k:\cc\times\rr\to\cc$ of Proposition~\ref{prop:AkandBk}, we have $f(z+\beta J)=A(z,\beta^2)+J\beta B(z,\beta^2)$ with
\begin{align*}
A(z,\gamma)&=\sum_{k\in\nn}\sum_{|\k|=k}A_\k(z-z_0,\gamma)c_\k\,,\\
B(z,\gamma)&=\sum_{k\in\nn}\sum_{|\k|=k}B_\k(z-z_0,\gamma)c_\k\,.
\end{align*}
 Here, absolute and uniform convergence are guaranteed by the inequalities $|A_\k(x_0+ix_1-z_0,x_2^2+x_3^2)|,\ |B_\k(x_0+ix_1-z_0,x_2^2+x_3^2)|\leq|\T_\k(x-z_0)|$, valid for all $x=x_0+ix_1+jx_2+kx_3\in B$.
\end{example}

\begin{remark}\label{rmk:whitney}
With the notations set in Theorem~\ref{thm:preslicefunction}, let $D':=\{(z,\beta^2)\in\cc\times\rr : (z,\beta)\in D\}$. By Whitney's Theorem, see~\cite{whitney}, there exist an open neighborhood $W$ of $D'$ in $\cc\times\rr$, with $W\cap(\cc\times\rr^\geqslant)=D'$, and real analytic functions $A,B:W\to\hh$ such that $F_\emptyset(z,\beta)=A(z,\beta^2)$ and $F_1(z,\beta)=\beta B(z,\beta^2)$, for all $(z,\beta)\in D$.
\end{remark}





\begin{thebibliography}{10}

\bibitem{librosommen}
F.~Brackx, R.~Delanghe, and F.~Sommen.
\newblock {\em Clifford analysis}, volume~76 of {\em Research Notes in
  Mathematics}.
\newblock Pitman (Advanced Publishing Program), Boston, MA, 1982.

\bibitem{librocnops}
J.~Cnops and H.~Malonek.
\newblock {\em An introduction to {C}lifford analysis}.
\newblock Textos de Matem{\'a}tica. S{\'e}rie B [Texts in Mathematics. Series
  B], 7. Universidade de Coimbra Departamento de Matem{\'a}tica, Coimbra, 1995.

\bibitem{israel}
F.~Colombo, I.~Sabadini, and D.~C. Struppa.
\newblock Slice monogenic functions.
\newblock {\em Israel J. Math.}, 171:385--403, 2009.

\bibitem{librodaniele2}
F.~Colombo, I.~Sabadini, and D.~C. Struppa.
\newblock {\em Noncommutative functional calculus. Theory and applications of
  slice hyperholomorphic functions}, volume 289 of {\em Progress in
  Mathematics}.
\newblock Birkh{\"a}user/Springer Basel AG, Basel, 2011.

\bibitem{sce}
P.~Dentoni and M.~Sce.
\newblock Funzioni regolari nell'algebra di {C}ayley.
\newblock {\em Rend. Sem. Mat. Univ. Padova}, 50:251--267 (1974), 1973.

\bibitem{ebbinghaus}
H.-D. Ebbinghaus, H.~Hermes, F.~Hirzebruch, M.~Koecher, K.~Mainzer,
  J.~Neukirch, A.~Prestel, and R.~Remmert.
\newblock {\em Numbers}, volume 123 of {\em Graduate Texts in Mathematics}.
\newblock Springer-Verlag, New York, 1990.
\newblock With an introduction by K. Lamotke, Translated from the second German
  edition by H. L. S. Orde, Translation edited and with a preface by J. H.
  Ewing, Readings in Mathematics.

\bibitem{fueter1}
R.~Fueter.
\newblock Die {F}unktionentheorie der {D}ifferentialgleichungen {$\Delta u=0$}
  und {$\Delta\Delta u=0$} mit vier reellen {V}ariablen.
\newblock {\em Comment. Math. Helv.}, 7(1):307--330, 1934.

\bibitem{fueter2}
R.~Fueter.
\newblock \"{U}ber die analytische {D}arstellung der regul\"aren {F}unktionen
  einer {Q}uaternionenvariablen.
\newblock {\em Comment. Math. Helv.}, 8(1):371--378, 1935.

\bibitem{librospringer2}
G.~Gentili, C.~Stoppato, and D.~C. Struppa.
\newblock {\em Regular functions of a quaternionic variable}.
\newblock Springer Monographs in Mathematics. Springer, Cham, [2022] \copyright
  2022.
\newblock Second edition [of 3013643].

\bibitem{cras}
G.~Gentili and D.~C. Struppa.
\newblock A new approach to {C}ullen-regular functions of a quaternionic
  variable.
\newblock {\em C. R. Math. Acad. Sci. Paris}, 342(10):741--744, 2006.

\bibitem{advances}
G.~Gentili and D.~C. Struppa.
\newblock A new theory of regular functions of a quaternionic variable.
\newblock {\em Adv. Math.}, 216(1):279--301, 2007.

\bibitem{rocky}
G.~Gentili and D.~C. Struppa.
\newblock Regular functions on the space of {C}ayley numbers.
\newblock {\em Rocky Mountain J. Math.}, 40(1):225--241, 2010.

\bibitem{perotti}
R.~Ghiloni and A.~Perotti.
\newblock Slice regular functions on real alternative algebras.
\newblock {\em Adv. Math.}, 226(2):1662--1691, 2011.

\bibitem{volumeintegral}
R.~Ghiloni and A.~Perotti.
\newblock Volume {C}auchy formulas for slice functions on real associative
  *-algebras.
\newblock {\em Complex Var. Elliptic Equ.}, 58(12):1701--1714, 2013.

\bibitem{gpsalgebra}
R.~Ghiloni, A.~Perotti, and C.~Stoppato.
\newblock The algebra of slice functions.
\newblock {\em Trans. Amer. Math. Soc.}, 369(7):4725--4762, 2017.

\bibitem{librogurlebeck2}
K.~G{{\"u}}rlebeck, K.~Habetha, and W.~Spr{{\"o}}{\ss}ig.
\newblock {\em Holomorphic functions in the plane and {$n$}-dimensional space}.
\newblock Birkh{\"a}user Verlag, Basel, 2008.
\newblock Translated from the 2006 German original, With 1 CD-ROM (Windows and
  UNIX).

\bibitem{krausshardifferentialtopological}
R.~S. Krau{\ss}har.
\newblock Differential topological aspects in octonionic monogenic function
  theory.
\newblock {\em Adv. Appl. Clifford Algebr.}, 30(4):Paper No. 51, 25, 2020.

\bibitem{moisilteodorescu}
G.~C. Moisil and N.~Teodorescu.
\newblock Fonctions holomorphes dans l'espace.
\newblock {\em Mathematica (Cluj)}, 5:142--159, 1931.

\bibitem{perotticr}
A.~Perotti.
\newblock Cauchy-{R}iemann operators and local slice analysis over real
  alternative algebras.
\newblock {\em J. Math. Anal. Appl.}, 516(1):Paper No. 126480, 34, 2022.

\bibitem{sudbery}
A.~Sudbery.
\newblock Quaternionic analysis.
\newblock {\em Math. Proc. Cambridge Philos. Soc.}, 85(2):199--224, 1979.

\bibitem{whitney}
H.~Whitney.
\newblock Differentiable even functions.
\newblock {\em Duke Math. J.}, 10:159--160, 1943.

\end{thebibliography}
\end{document}